\documentclass{amsart}

\newtheorem{theorem}{Theorem}
\newtheorem{lemma}[theorem]{Lemma}

\theoremstyle{definition}
\newtheorem{definition}[theorem]{Definition}

\theoremstyle{remark}

\numberwithin{equation}{section}

\begin{document}
\title{Generalized $(m,n)$-Jordan centralizers and derivations on semiprime rings}


\author{Arindam Ghosh}
\address{Department of Mathematics, Indian Institute of Technology Patna, Patna-801106}
\curraddr{}
\email{E-mail: arindam.pma14@iitp.ac.in}
\thanks{}

\author{Om Prakash$^{\star}$}
\address{Department of Mathematics, Indian Institute of Technology Patna, Patna-801106}
\curraddr{}
\email{om@iitp.ac.in}
\thanks{* Corresponding author}

\subjclass[2010]{16N60, 16W25, 39B05.}

\keywords{prime ring, semiprime ring, $(m; n)$-Jordan centralizer, generalized $(m; n)$-Jordan centralizer, $(m; n)$-Jordan derivation, generalized $(m; n)$-Jordan derivation.}

\date{}

\dedicatory{}

\begin{abstract}
In this article, we prove Conjecture $1$ posed in 2013 by Fo$\check{s}$ner \cite{fos} and Conjecture $1$ posed in 2014 by Ali and Fo$\check{s}$ner \cite{ali} related to generalized $(m,n)$-Jordan centralizer and derivation respectively.
\end{abstract}

\maketitle

Throughout the paper $R$ represents an associative ring, $Z(R)$ the center of the ring $R$. A ring $R$ is said to be a prime ring if $aRb=0$ for some $a,b\in R$ implies either $a=0$ or $b=0$ and is said to be a semiprime ring if $aRa=0$ for some $a\in R$ implies $a=0$. The ring $R$ is $n$-torsion free if $na=0$ for some $a\in R$ implies $a=0$, where $n\geq 2$ is an integer. An additive map $T:R\rightarrow R$ is a left (right) centralizer if $T(xy)=T(x)y$ ($T(xy)=xT(y)$), for all $x,y\in R$. If $R$ has unity $1\neq 0$ and $T:R\rightarrow R$ is a left (right) centralizer, then $T(x)=T(1)x~(T(x)=xT(1))$, for all $x\in R$. An additive map $T:R\rightarrow R$ is a two-sided centralizer if $T(xy)=T(x)y=xT(y)$, for all $x,y\in R$. Also, an additive map $T:R\rightarrow R$ is said to be a left (right) Jordan centralizer if $T(x^2)=T(x)x$ ($T(x^2)=xT(x)$), for all $x\in R$. We denote $[x,y]$ by $xy-yx$. A mapping $T:R\rightarrow R$ is said to be centralizing on $R$ if $[T(x),x]\in Z(R)$, for all $x\in R$ and is said to be commuting on $R$ if $[T(x),x]=0$, for all $x\in R$.\\

\begin{lemma}
[\cite{vukma}, Lemma 1]
\label{new2}
Let $R$ be a semiprime ring and $axb + bxc = 0$, for all $x \in R$ and some $a, b, c \in R$. Then $(a + c)xb = 0$, for all $x \in R$.
\end{lemma}

\begin{theorem}
[\cite{vukman}, Theorem 4]
\label{new3}
Let $R$ be a $2$-torsion free semiprime ring. If an additive mapping $T : R \rightarrow R$ satisfies the relation $[[T(x), x], x] = 0$ for all $x \in R$, then $T$ is commuting on $R$.
\end{theorem}

\section{Generalized $(m,n)$-Jordan centralizers}
In 1991, Zalar \cite{zal} established that every left (right) Jordan centralizer over a $2$-torsion free semiprime ring is a left (right) centralizer. Also, Theorem 2.3.2 of \cite{bei} says that every two-sided centralizer $T$ over a semiprime ring $R$ with extended centroid $C$ is of the form $T(x)=\lambda x$, for all $x\in R$ and for some $\lambda \in C$. In 1995, Moln{\'a}r \cite{mol} has shown that an additive mapping $T$ over semisimple $H^{\star}$-algebra $A$ with $T (x^3) = T (x)x^2 ~(T (x^3) = x^2T (x))$ for all $x \in A$, is a left (right) centralizer. Further, some results related to semisimple $H^{\star}$-algebras can be found in \cite{ulbl,vu,vuk}. In 2010, Vukman \cite{vukm} introduced a new kind of map as follows:\\

\begin{definition}
Let $m \geq 0, n \geq 0$ with $m + n \neq 0$ be some fixed integers and $R$ be a ring. An additive mapping $T : R \rightarrow R$ is said to be a $(m; n)$-Jordan centralizer if
\begin{equation}
\label{eq:mn1}
(m  + n)T(x^2) = mT(x)x + nxT(x),~\text{for all}~ x \in R.
\end{equation}
\end{definition}

It can easily be seen that $(1;0)$-Jordan centralizer is a left Jordan centralizer and $(0;1)$-Jordan centralizer is a right Jordan centralizer. In 1999, Vukman \cite{v} proved that an additive map $T$ over a $2$-torsion free semiprime ring $R$ satisfying $2T(x^2)=T(x)x+xT(x)$, for all $x\in R$ (that is, a $(1;1)$-Jordan centralizer) is a two-sided centralizer. So, we can easily conclude that every $(m;m)$-Jordan centralizer over $2m$-torsion free semiprime ring is a two-sided centralizer, for $m\geq 1$. Again, in 2010, Vukman \cite{vukm} proved that every $(m;n)$-Jordan centralizer on $6mn(m + n)$-torsion free prime ring with nonzero center is a two-sided centralizer. Later, in 2016, Kosi-Ulbl and Vukman \cite{ulb} have proved that every $(m;n)$-Jordan centralizer on a $mn(m+n)$-torsion free semiprime ring is a two-sided centralizer. Meanwhile, in 2013, Fo\v{s}ner \cite{fos} introduced the concept of generalized $(m; n)$-Jordan centralizers over a ring.\\

\begin{definition}
Let $m \geq 0, n \geq 0$ with $m + n \neq 0$ be some fixed integers and $R$ be a ring. An additive mapping $T : R \rightarrow R$ is said to be a generalized $(m; n)$-Jordan centralizer if there exists an $(m; n)$-Jordan centralizer $T_0 : R \rightarrow R$ such that
\begin{equation}
\label{eq:mn2}
(m  + n)T(x^2) = mT(x)x + nxT_0(x),~\text{for all}~ x \in R.
\end{equation}
\end{definition}

It is obvious that generalized $(1;0)$-Jordan centralizer is a left Jordan centralizer. Fo\v{s}ner \cite{fos} proved that every generalized $(m; n)$-Jordan centralizer on a $6mn(m + n)(m + 2n)$-torsion free prime ring with nonzero center is a two-sided centralizer. Also, he conjectured the result for semiprime ring as Conjecture 1 \cite{fos}. Here, we present the proof of that conjecture as Theorem \ref{mn3}. Before proving the theorem, we have several lemmas.\\

\begin{lemma}
[\cite{vukm}, Proposition 3]
\label{new1}
Let $m \geq 0, n \geq 0$ with $m + n \neq 0$ be some fixed integers, $R$ be a ring and $T : R \rightarrow R$ be a $(m; n)$-Jordan centralizer. Then
\begin{equation}
\label{eq:new1a}
\begin{aligned}
 &2(m + n)^2T(xyx) = mnT(x)xy + m(2m + n)T(x)yx - mnT(y)x^2\\
&+ 2mnxT(y)x - mnx^2T(y) + n(m + 2n)xyT(x) + mnyxT(x),~\text{for all}~x,y\in R.
\end{aligned}
\end{equation}
\end{lemma}

\begin{lemma}
[\cite{fos}, Lemma 1]
\label{mn1}
Let $m \geq 0, n \geq 0$ with $m + n \neq 0$ be some fixed integers, $R$ be a ring and $T : R \rightarrow R$ be a generalized $(m; n)$-Jordan centralizer. Then for $(m; n)$-Jordan centralizer $T_0$ in \eqref{eq:mn2},
\begin{equation}
\label{eq:mn3}
\begin{aligned}
 &2(m + n)^2T(xyx) = mnT(x)xy + m(2m + n)T(x)yx - mnT(y)x^2\\
&+ 2mnxT_0(y)x - mnx^2T_0(y) + n(m + 2n)xyT_0(x) + mnyxT_0(x),~\text{for all}~x,y\in R.
\end{aligned}
\end{equation}
\end{lemma}

\begin{theorem}
[\cite{ulb}, Theorem 2]
\label{mn2}
Let $m \geq 1$, $n \geq 1$ be some fixed integers, $R$ be a $mn(m+n)$-torsion free semiprime ring and $T : R \rightarrow R$ be an $(m; n)$-Jordan centralizer. In this case $T$ is a two-sided centralizer.
\end{theorem}

Now, we present the conjecture posed by Fo$\check{s}$ner in \cite{fos}:\\
\textbf{Conjecture 1.} Let $m \geq 1, n \geq 1$ be some fixed integers, $R$ be a semiprime ring with suitable torsion restrictions and $T : R \rightarrow R$ be a generalized $(m; n)$-Jordan centralizer. Then $T$ is a two-sided centralizer.

Now, we are in a state to prove the main Theorem.
\begin{theorem}
\label{mn3}
Let $m \geq 1$, $n \geq 1$ be some fixed integers, $R$ be a $mn(m+n)(2m+n)$-torsion free semiprime ring and $T : R \rightarrow R$ be a generalized $(m, n)$-Jordan centralizer. Then $T$ is a two-sided centralizer.
\end{theorem}

\begin{proof}
Since $T : R \rightarrow R$ be a generalized $(m, n)$-Jordan centralizer, so it satisfies \eqref{eq:mn2} for some $(m, n)$-Jordan centralizer $T_0$. Since $R$ is semiprime, by Theorem \ref{mn2}, $T_0$ is a two-sided centralizer. So, $T_0(xy)=T_0(x)y=xT_0(y)$, for all $x,y \in R$. we frequently use this and $2$-torsion free condition of $R$ without mentioning. Now, from Lemma \ref{mn1}, $T$ satisfies
\begin{equation}
\label{eq:new1}
\begin{aligned}
 &2(m + n)^2T(xyx) = mnT(x)xy + m(2m + n)T(x)yx - mnT(y)x^2\\
&- mnT_0(x^2y) + n(3m + 2n)T_0(xyx) + mnT_0(yx^2),~\text{for all}~x,y\in R.
\end{aligned}
\end{equation}

Replacing $x$ by $x+y$ in \eqref{eq:mn2}, we have
\begin{equation}
\label{eq:new2}
\begin{aligned}
 &(m + n)T(xy + yx) = mT(x)y + mT(y)x + nT_0(xy+yx),~\text{for all}~x,y\in R.
\end{aligned}
\end{equation}

Put $y=(m+n)(xy+yx)$ in \eqref{eq:new1} and applying \eqref{eq:new2}, we get
\begin{equation}
\label{eq:new3}
\begin{aligned}
&2(m + n)^3T(x^2yx + xyx^2) = mn(m + n)T(x)x^2y + 2m(m + n)^2T(x)xyx\\
&+m(2m^2 + 2mn + n^2)T(x)yx^2 - m^2nT(y)x^3 + mn(2m - n)xT_0(y)x^2\\
&-mn^2yT_0(x)x^2 + 2m^2nxT_0(x)yx + mn(2n - m)x^2T_0(y)x + 2mn^2xyT_0(x)x\\
&-m^2nx^2T_0(x)y - mn^2x^3T_0(y) + n(2n^2 + 2mn + m^2)x^2yT_0(x)\\
&+2n(m + n)^2xyxT_0(x) + mn(m + n)yx^2T_0(x),~\text{for all}~x,y\in R.
\end{aligned}
\end{equation}

Put $y=2(m+n)^2xyx$ in \eqref{eq:new2} and apply \eqref{eq:new1}, we have
\begin{equation}
\label{eq:new4}
\begin{aligned}
&2(m + n)^3T(x^2yx + xyx^2) = m(2m^2+5mn + 2n^2)T(x)xyx + m^2(2m+n)T(x)yx^2\\
&-m^2nT(y)x^3 + mn(2m - n)xT_0(y)x^2 + mn(2n - m)x^2T_0(y)x + mn(2n + m)xyT_0(x)x\\
&+m^2nyxT_0(x)x + mn^2xT_0(x)xy + mn(2m + n)xT_0(x)yx - mn^2x^3T_0(y)\\
&+n^2(2n + m)x^2yT_0(x) + n(2m^2+5mn + 2n^2)xyxT_0(x),~\text{for all}~x,y\in R.
\end{aligned}
\end{equation}

Comparing \eqref{eq:new3} and \eqref{eq:new4}, we get
\begin{equation}
\label{eq:new5}
\begin{aligned}
&(m + n)T(x)x^2y - mT(x)xyx + (m + n)T(x)yx^2 - nyT_0(x)x^2\\
&-nxT_0(x)yx - mxyT_0(x)x - mx^2T_0(x)y + (m + n)x^2yT_0(x) - nxyxT_0(x)\\
&+(m + n)yx^2T_0(x) - myxT_0(x)x - nxT_0(x)xy = 0,~\text{for all}~x,y\in R.
\end{aligned}
\end{equation}

Replacing $y$ by $xy$ in \eqref{eq:new5}, we have
\begin{equation}
\label{eq:new6}
\begin{aligned}
&(m + n)T(x)x^2xy - mT(x)xxyx + (m + n)T(x)xyx^2 - nxyT_0(x)x^2\\
&-nxT_0(x)xyx - mxxyT_0(x)x - mx^2T_0(x)xy + (m + n)x^2xyT_0(x) - nxxyxT_0(x)\\
&+(m + n)xyx^2T_0(x) - mxyxT_0(x)x - nxT_0(x)xxy = 0,~\text{for all}~x,y\in R.
\end{aligned}
\end{equation}

Multiplying \eqref{eq:new5} by $x$ from left, we get
\begin{equation}
\label{eq:new7}
\begin{aligned}
&(m + n)xT(x)x^2y - mxT(x)xyx + (m + n)xT(x)yx^2 - nxyT_0(x)x^2\\
&-nxxT_0(x)yx - mxxyT_0(x)x - mxx^2T_0(x)y + (m + n)xx^2yT_0(x) - nxxyxT_0(x)\\
&+(m + n)xyx^2T_0(x) - mxyxT_0(x)x - nxxT_0(x)xy = 0,~\text{for all}~x,y\in R.
\end{aligned}
\end{equation}

Now, subtract \eqref{eq:new7} from \eqref{eq:new6},
\begin{equation}
\label{eq:new9}
\begin{aligned}
 (m+n)[T(x),x]x^2y-m[T(x),x]xyx+(m+n)[T(x),x]yx^2= 0,~\text{for all}~x,y\in R.
\end{aligned}
\end{equation}

Put $y=yT(x)$ in \eqref{eq:new9}, we have
\begin{equation}
\label{eq:new10}
\begin{aligned}
 &(m+n)[T(x),x]x^2yT(x)-m[T(x),x]xyT(x)x+(m+n)[T(x),x]yT(x)x^2= 0,\\
 &~\text{for all}~x,y\in R.
\end{aligned}
\end{equation}

Multiplying \eqref{eq:new9} by $T(x)$ from right, we get
\begin{equation}
\label{eq:new11}
\begin{aligned}
 &(m+n)[T(x),x]x^2yT(x)-m[T(x),x]xyxT(x)+(m+n)[T(x),x]yx^2T(x)= 0,\\
 &~\text{for all}~x,y\in R.
\end{aligned}
\end{equation}

Again, subtract \eqref{eq:new11} from \eqref{eq:new10}, we have
\begin{equation}
\label{eq:new12}
\begin{aligned}
&-m[T(x),x]xy[T(x),x]+(m+n)[T(x),x]y[T(x),x^2]=0,\\
&~\text{for all}~x,y\in R~\text{(since $R$ is $2$-torsion free)}.
\end{aligned}
\end{equation}

By using Lemma \ref{new2} on \eqref{eq:new12}, we have
\begin{equation}
\label{eq:new13}
\begin{aligned}
&(-m[T(x),x]x+(m+n)[T(x),x^2])y[T(x),x]=0\\
& \implies (mx[T(x),x]+n[T(x),x^2])y[T(x),x]=0 ,~\text{for all}~x,y\in R.
\end{aligned}
\end{equation}

Since $[T(x),x^2]=[T(x),x]x+x[T(x),x]$, \eqref{eq:new13} reduces to
\begin{equation}
\label{eq:new14}
\begin{aligned}
((m+n)x[T(x),x]+n[T(x),x]x)y[T(x),x]=0 ,~\text{for all}~x,y\in R.
\end{aligned}
\end{equation}

Therefore, right multiplication of \eqref{eq:new14} by $nx$, gives us
\begin{equation}
\label{eq:new15}
\begin{aligned}
((m+n)x[T(x),x]+n[T(x),x]x)y[T(x),x]nx=0 ,~\text{for all}~x,y\in R.
\end{aligned}
\end{equation}

Replacing $y$ by $(m+n)yx$ in \eqref{eq:new14},
\begin{equation}
\label{eq:new16}
\begin{aligned}
((m+n)x[T(x),x]+n[T(x),x]x)(m+n)yx[T(x),x]=0 ,~\text{for all}~x,y\in R.
\end{aligned}
\end{equation}

Adding \eqref{eq:new15} and \eqref{eq:new16}, we get
\begin{equation}
\label{eq:new17}
\begin{aligned}
&((m+n)x[T(x),x]+n[T(x),x]x)y((m+n)x[T(x),x]+n[T(x),x]x)=0 ,\\
&~\text{for all}~x,y\in R.
\end{aligned}
\end{equation}

Since $R$ is semiprime,
\begin{equation}
\label{eq:new18}
\begin{aligned}
&(m+n)x[T(x),x]+n[T(x),x]x=0\\
& \implies nT(x)x^2+mxT(x)x-(m+n)x^2T(x)=0,~\text{for all}~x\in R.
\end{aligned}
\end{equation}

Putting $x=x+y$ in \eqref{eq:new18},
\begin{equation}
\label{eq:new19}
\begin{aligned}
&nT(x)(xy+yx+y^2)+nT(y)(x^2+xy+yx)+mxT(x)y+mxT(y)(x+y)\\
&+myT(x)(x+y)+myT(y)x-(m+n)[(xy+yx+y^2)T(x)\\
&+(x^2+xy+yx)T(y)]=0,~\text{for all}~x,y\in R.
\end{aligned}
\end{equation}

Put $x=-x$ in \eqref{eq:new19} and then adding it to \eqref{eq:new19}, we have
\begin{equation}
\label{eq:new20}
\begin{aligned}
&nT(x)(xy+yx)+nT(y)x^2+mxT(x)y+mxT(y)x+myT(x)x\\
&-(m+n)[(xy+yx)T(x)+x^2T(y)]=0,~\text{for all}~x,y\in R~\text{(since $R$ is $2$-torsion free)}.
\end{aligned}
\end{equation}

Put $y=(m+n)(xy+yx)$ in \eqref{eq:new20} and using \eqref{eq:new2},
\begin{equation}
\label{eq:new20a}
\begin{aligned}
&n(m+n)T(x)[x^2y+2xyx+yx^2]+mn[T(x)y+T(y)x]x^2+n^2T_0(xyx^2+yx^3)\\
&+m(m+n)xT(x)(xy+yx)+m^2x(T(x)y+T(y)x)x+mnT_0(x^2yx+xyx^2)\\
&+m(m+n)(xy+yx)T(x)x-(m+n)^2(x^2y+2xyx+yx^2)T(x)\\
&-m(m+n)x^2(T(x)y+T(y)x)-n(m+n)T_0(x^3y+x^2yx)=0,~\text{for all}~x,y\in R.
\end{aligned}
\end{equation}

Using \eqref{eq:new18} in \eqref{eq:new20a} and rearranging,
\begin{equation}
\label{eq:new20b}
\begin{aligned}
&n(m+n)x^2T(x)y-n(m+n)yT(x)x^2+2n(m+n)T(x)xyx+n(2m+n)T(x)yx^2\\
&+nmT(y)x^3+(m^2+mn+n^2)xT(x)yx+m^2xT(y)x^2+m(m+n)xyT(x)x\\
&-(m+n)^2 x^2 yT(x)-2(m+n)^2xyxT(x)-m(m+n)x^2T(y)x+n^2T_0(xyx^2+yx^3)\\
&+nmT_0(x^2yx+xyx^2)-n(m+n)T_0(x^3y+x^2yx)=0,~\text{for all}~x,y\in R.
\end{aligned}
\end{equation}

Using \eqref{eq:new20} in \eqref{eq:new20b} and rearranging,
\begin{equation}
\label{eq:new21}
\begin{aligned}
&n(m+n)x^2T(x)y-(m^2+mn+n^2)yT(x)x^2+n(m+2n)T(x)xyx\\
&+n(m+n)T(x)yx^2+n(m+n)xT(x)yx+2m(m+n)xyT(x)x-(m+n)^2 x^2 yT(x)\\
&+m(m+n)yxT(x)x+n^2T_0(xyx^2+yx^3)+nmT_0(x^2yx+xyx^2)\\
&-2(m+n)^2xyxT(x)-n(m+n)T_0(x^3y+x^2yx)=0,~\text{for all}~x,y\in R.
\end{aligned}
\end{equation}

Putting $y=xy$ in \eqref{eq:new21} and also multiplying \eqref{eq:new21} by $x$ from left and after subtracting these new equations, we get
\begin{equation}
\label{eq:new22}
\begin{aligned}
&n(m+n)x^2[T(x),x]y+n(m+2n)[T(x),x]xyx\\
&+n(m+n)[T(x),x]yx^2+n(m+n)x[T(x),x]yx=0,~\text{for all}~x,y\in R.
\end{aligned}
\end{equation}

Also, replacing $y$ by $yT(x)$ in \eqref{eq:new22} and again multiplying \eqref{eq:new22} by $T(x)$ from right, and finally subtracting these obtained relations, we get
\begin{equation}
\label{eq:new23}
\begin{aligned}
&n(m+2n)[T(x),x]xy[T(x),x]+n(m+n)[T(x),x]y[T(x),x^2]\\
&+n(m+n)x[T(x),x]y[T(x),x]=0,~\text{for all}~x,y\in R.
\end{aligned}
\end{equation}

Using \eqref{eq:new18} in \eqref{eq:new23},
\begin{equation}
\label{eq:new23a}
\begin{aligned}
&n(m+n)[T(x),x]xy[T(x),x]+n(m+n)[T(x),x]y[T(x),x^2]=0,~\text{for all}~x,y\in R.
\end{aligned}
\end{equation}

Now since
\begin{equation}
\label{eq:new23b}
\begin{aligned}
&(m+n)[T(x),x^2]=m[T(x),x]x+((m+n)x[T(x),x]+n[T(x),x]x)=m[T(x),x]x,\\
&~\text{for all}~x\in R~(\text{by} ~\eqref{eq:new18}).
\end{aligned}
\end{equation}

From \eqref{eq:new23a} and \eqref{eq:new23b},
\begin{equation}
\label{eq:new23c}
\begin{aligned}
&n(m+n)[T(x),x]xy[T(x),x]+nm[T(x),x]y[T(x),x]x=0,~\text{for all}~x,y\in R.
\end{aligned}
\end{equation}

Applying Lemma \ref{new2} on \eqref{eq:new23c},
\begin{equation}
\label{eq:new24}
\begin{aligned}
&[T(x),x]xy[T(x),x]=0,~\text{for all}~x,y\in R~(\text{since $R$ is $n(2m+n)$-torsion free}).
\end{aligned}
\end{equation}

Using \eqref{eq:new18} in \eqref{eq:new24},
\begin{equation}
\label{eq:new25}
\begin{aligned}
&x[T(x),x]y[T(x),x]=0,~\text{for all}~x,y\in R~(\text{since $R$ is $(m+n)$-torsion free}).
\end{aligned}
\end{equation}

Subtracting \eqref{eq:new25} from \eqref{eq:new24}, we get
\begin{equation}
\label{eq:new26}
\begin{aligned}
&[[T(x),x],x]y[T(x),x]=0,~\text{for all}~x,y\in R.
\end{aligned}
\end{equation}

Finally, multiplying \eqref{eq:new26} by $x$ from right and putting $y=yx$ in \eqref{eq:new26}, and then subtracting these two, we get
\begin{equation}
\label{eq:new26a}
\begin{aligned}
&[[T(x),x],x]y[[T(x),x],x]=0~\text{for all}~x,y\in R\\
& \implies [[T(x),x],x]=0,~\text{for all}~x\in R~\text{(since $R$ is semiprime)}.
\end{aligned}
\end{equation}

By Theorem \ref{new3},
\begin{equation}
\label{eq:new27}
\begin{aligned}
&[T(x),x]=0\\
& \implies T(x)x=xT(x),~\text{for all}~x\in R~\text{(since $R$ is semiprime)}.
\end{aligned}
\end{equation}

Let $F=T-T_0$. Then by \eqref{eq:new27}, $F$ is an additive mapping satisfying
\begin{equation}
\label{eq:new29}
\begin{aligned}
F(x)x=xF(x)=0,~\text{for all}~x\in R.
\end{aligned}
\end{equation}

Taking $x=x+y$ in \eqref{eq:new29},
\begin{equation}
\label{eq:new30}
\begin{aligned}
F(x)y+F(y)x=0,~\text{for all}~x,y\in R.
\end{aligned}
\end{equation}

Multiplying \eqref{eq:new30} by $F(x)$ from right,
\begin{equation}
\label{eq:new31}
\begin{aligned}
&F(x)yF(x)=0, ~\text{for all}~x,y\in R \\
&\implies F(x)=0,~\text{for all}~x\in R~\text{(since $R$ is semiprime)}.
\end{aligned}
\end{equation}

Hence $T=T_0$. Thus $T$ is a two-sided centralizer.
\end{proof}

\section{Generalized $(m,n)$-Jordan derivations}
An additive map $D:R\rightarrow R$ is said to be a derivation if $D(xy)=D(x)y+xD(y)$, for all $x,y\in R$ and is said to be a Jordan derivation if $D(x^2)=D(x)x+xD(x)$, for all $x\in R$. In fact, Herstein introduced Jordan derivation over rings in 1957 and he proved that every Jordan derivation over a $2$-torsion free prime ring is a derivation (\cite{her}, Theorem 3.2). In 1975, Cusack generalized the result for semiprime rings (Corollary 5, \cite{cus}). In 1988, Bre\v{s}ar gives the proof of Cusack's result in a new way (Theorem 1, \cite{bre}). To see more results on Jordan derivation, we refer \cite{gho,zha,zhan}. \\

 In 1990, Bre\v{s}ar and Vukman introduced left derivation and Jordan left derivation as an additive map $D:R\rightarrow R$ is to be a left derivation if $D(xy)=xD(y)+yD(x)$, for all $x,y\in R$ and a Jordan left derivation if $D(x^2)=2xD(x)$, for all $x\in R$. They proved that the existence of a nonzero Jordan left derivation of a prime ring of characteristic $\neq 2,3$ forces the ring to be commutative. In 1992 \cite{den}, Deng shown that there is no need of the assumption that ring to be of characteristic not $3$. In 2008 \cite{v'}, Vukman proved that every left Jordan derivation over a $2$-torsion free semiprime ring is a derivation which maps the ring into its center. More related results can be seen in \cite{ash, ghos, xu}.\\
 
 The concept of generalized derivation was introduced by Bre\v{s}ar \cite{bresa} in 1991. An additive map $G:R\rightarrow R$ is said to be a generalized derivation if there exists a derivation $D:R\rightarrow R$ such that $G(xy)=G(x)y+xD(y)$, for all $x,y\in R$ and is said to be a generalized Jordan derivation if there exists a Jordan derivation $D:R\rightarrow R$ such that $G(x^2)=G(x)x+xD(x)$, for all $x\in R$. Note that every generalized derivation over a ring is a sum of a derivation and a left centralizer and the sum is unique for semiprime rings. For more related results for generalized derivations, we refer \cite{hva,vuk'}.\\

In 2008, Vukman \cite{vu'} introduced the concept of $(m; n)$-Jordan derivation which is as follows:

\begin{definition}
Let $m \geq 0, n \geq 0$ with $m + n \neq 0$ be some fixed integers and $R$ be a ring. An additive mapping $D : R \rightarrow R$ is said to be a $(m; n)$-Jordan derivation if
\begin{equation}
\label{eq:mnd1}
(m  + n)D(x^2) =2 mD(x)x + 2nxD(x),~\text{for all}~ x \in R.
\end{equation}
\end{definition}

It is obvious that every $(1; 0)$-Jordan derivation is a Jordan left derivation and  every $(1; 1)$-Jordan derivation over a $2$-torsion free ring is a Jordan derivation. Let $m , n \geq 1$ be integers with $m\neq n$. In 2008, Vukman \cite{vu'} proved that every $(m; n)$-Jordan derivation over a prime ring of characteristic not equal to $2mn(m+n)|m-n|$ is a derivation and forces the ring to be commutative. Recently, in 2016, Ulbl and Vukman \cite{ul} proved that every $(m; n)$-Jordan derivation over a $mn(m+n)|m-n|$-torsion free semiprime ring is a derivation which maps the ring into its center. \\
In 2014, Ali and Fo\v{s}ner \cite{ali} introduced the concept of generalized $(m; n)$-Jordan derivation.

\begin{definition}
Let $m \geq 0, n \geq 0$ with $m + n \neq 0$ be some fixed integers and $R$ be a ring. An additive mapping $F : R \rightarrow R$ is said to be a generalized $(m; n)$-Jordan derivation if there exists an $(m; n)$-Jordan derivation $D : R \rightarrow R$ such that
\begin{equation}
\label{eq:mnd2}
(m  + n)F(x^2) = 2mF(x)x + 2nxD(x),~\text{for all}~ x \in R.
\end{equation}
\end{definition}

It is obvious that every generalized $(1; 1)$-Jordan derivation over a $2$-torsion free ring is a generalized Jordan derivation. Let $m , n \geq 1$ be integers with $m\neq n$. In \cite{ali}, they proved that every generalized $(m; n)$-Jordan derivation over a prime ring of characteristic not equal to $6mn(m+n)|m-n|$ is a derivation and forces the ring to be commutative. They also conjectured this result for semiprime ring given as Conjecture 2.

\textbf{Conjecture 2.} Let $m \geq 1, n \geq 1$ be some fixed integers, $R$ be a semiprime ring with suitable torsion restrictions and $F : R \rightarrow R$ be a generalized $(m; n)$-Jordan derivation. Then $F$ is a derivation which maps $R$ into $Z(R)$.

Before proving the conjecture, we have several lemmas.
\begin{lemma}
[\cite{ali}, Lemma 1]
\label{mnd1}
Let $m \geq 0, n \geq 0$ with $m + n \neq 0$ be some fixed integers, $R$ be a $2$-torsion free ring and $F : R \rightarrow R$ be a generalized $(m; n)$-Jordan derivation. Then for $(m; n)$-Jordan derivation $D$ in \eqref{eq:mnd2},
\begin{equation}
\label{eq:mnd3}
\begin{aligned}
 &(m + n)^2F(xyx) =  m(n - m)F(x)xy + m(m - n)F(y)x^2 + n(n - m)x^2D(y)\\
&+ n(m - n)yxD(x) + m(3m + n)F(x)yx + 4mnxD(y)x+ n(3n + m)xyD(x),\\
&~\text{for all}~x,y\in R.
\end{aligned}
\end{equation}
\end{lemma}

\begin{theorem}
[\cite{ul}, Theorem 1.5]
\label{mnd3}
Let $m \geq 1$, $n \geq 1$ be distinct integers, $R$ be a $mn(m+n)|m-n|$-torsion free semiprime ring and $D : R \rightarrow R$ be an $(m; n)$-Jordan derivation. Then $D$ is a derivation which maps $R$ into $Z(R)$.
\end{theorem}

Now, we are in a state to prove the conjecture.
\begin{theorem}
\label{mnd5}
Let $m \geq 1$, $n \geq 1$ be distinc integers, $R$ be a $mn(m+n)|n-m|$-torsion free semiprime ring and $F : R \rightarrow R$ be a generalized $(m, n)$-Jordan derivation. Then $F$ is a derivation which maps $R$ into $Z(R)$.
\end{theorem}

\begin{proof}
Since $F : R \rightarrow R$ be a generalized $(m, n)$-Jordan derivation, so it satisfies \eqref{eq:mnd2} for some $(m, n)$-Jordan derivation $D$. Since $R$ is semiprime, by Theorem \ref{mnd3}, $D$ is a derivation which maps $R$ into $Z(R)$. Hence
\begin{equation}
\label{eq:newd1a}
\begin{aligned}
 &D(x)y=yD(x),\\
 &D(xy) = D(x)y+xD(y)=xD(y)+yD(x)~\text{for all}~x,y\in R.
\end{aligned}
\end{equation}

we frequently use \eqref{eq:newd1a} and $2$-torsion free condition of $R$ without mentioning. Now, from Lemma \ref{mnd1}, $F$ satisfies
\begin{equation}
\label{eq:newd1}
\begin{aligned}
 &(m + n)^2F(xyx) =  m(n - m)F(x)xy + m(m - n)F(y)x^2 + n(n - m)x^2D(y)\\
&+ n(m - n)yxD(x) + m(3m + n)F(x)yx + 4mnxD(y)x+ n(3n + m)xyD(x),\\
&~\text{for all}~x,y\in R.
\end{aligned}
\end{equation}

Replacing $x$ by $x+y$ in \eqref{eq:mnd2}, we have
\begin{equation}
\label{eq:newd2}
\begin{aligned}
 &(m + n)F(xy + yx) =2 mF(x)y + 2mF(y)x + 2nD(xy),~\text{for all}~x,y\in R.
\end{aligned}
\end{equation}

Put $y=(m+n)^2 xyx$ in \eqref{eq:newd2} and applying \eqref{eq:newd1}, we get
\begin{equation}
\label{eq:newd3}
\begin{aligned}
&(m + n)^3F(x^2yx + xyx^2) = 2m(3mn + n^2)F(x)xyx + 2m^2(m - n)F(y)x^3\\
&+ 2mn(5n - m)x^2D(y)x + 2mn(m - n)yxD(x)x + 2m^2(3m + n)F(x)yx^2\\
&+ 2mn(5m - n)xD(y)x^2 + 2mn(3n + m)xyD(x)x + 2mn(n - m)xD(x)xy\\
&+ 2n^2(n - m)x^3D(y) + 2mn(m + 3n)xyxD(x) + 2mn(3m + n)xD(x)yx\\
&+ 2n^2(3n + m)x^2yD(x),~\text{for all}~x,y\in R.
\end{aligned}
\end{equation}

Put $y=(m+n)(xy+yx)$ in \eqref{eq:newd1} and apply \eqref{eq:newd2}, we have
\begin{equation}
\label{eq:newd4}
\begin{aligned}
&(m + n)^3F(x^2yx + xyx^2) = m(m + n)(n - m)F(x)x^2y + 2m(m + n)^2F(x)xyx\\
&+m(5m^2 + 2mn + n^2)F(x)yx^2 + 2m^2(m - n)F(y)x^3 + 2mn(5m - n)xD(y)x^2\\
&+ 2mn(m - n)yD(x)x^2 + 2mn(n - m)x^2D(x)y + 2mn(5n - m)x^2D(y)x\\
&+ 2n^2(n - m)x^3D(y) + n(5n^2 + 2mn + m^2)x^2yD(x) + 2n(m + n)^2xyxD(x)\\
&+ n(m + n)(m - n)yx^2D(x) + 8m^2nxD(x)yx + 8mn^2xyD(x)x,~\text{for all}~x,y\in R.
\end{aligned}
\end{equation}

Comparing \eqref{eq:newd3} and \eqref{eq:newd4}, we get
\begin{equation}
\label{eq:newd5}
\begin{aligned}
&2m^2F(x)xyx - m(m + n)F(x)x^2y - m(m + n)F(x)yx^2 + 2mnyD(x)x^2\\
&- 2mnx^2D(x)y + n(m + n)x^2yD(x) - 2n^2xyxD(x) + n(m + n)yx^2D(x)\\
&- 2mnxyD(x)x + 2mnxD(x)yx - 2mnyxD(x)x + 2mnxD(x)xy = 0,\\
&~\text{for all}~x,y\in R~(\text{since $R$ is $|n-m|$-torsion free}).
\end{aligned}
\end{equation}

Replacing $y$ by $xy$ in \eqref{eq:newd5}, we have
\begin{equation}
\label{eq:newd6}
\begin{aligned}
&2m^2F(x)xxyx - m(m + n)F(x)x^2xy - m(m + n)F(x)xyx^2 + 2mnxyD(x)x^2\\
&- 2mnx^2D(x)xy + n(m + n)x^2xyD(x) - 2n^2xxyxD(x) + n(m + n)xyx^2D(x)\\
&- 2mnxxyD(x)x + 2mnxD(x)xyx - 2mnxyxD(x)x + 2mnxD(x)xxy = 0,\\
&~\text{for all}~x,y\in R.
\end{aligned}
\end{equation}

Multiplying \eqref{eq:newd5} by $x$ from left, we get
\begin{equation}
\label{eq:newd7}
\begin{aligned}
&2m^2xF(x)xyx - m(m + n)xF(x)x^2y - m(m + n)xF(x)yx^2 + 2mnxyD(x)x^2\\
&- 2mnxx^2D(x)y + n(m + n)xx^2yD(x) - 2n^2xxyxD(x) + n(m + n)xyx^2D(x)\\
&- 2mnxxyD(x)x + 2mnxxD(x)yx - 2mnxyxD(x)x + 2mnxxD(x)xy = 0,\\
&~\text{for all}~x,y\in R.
\end{aligned}
\end{equation}

Now, subtract \eqref{eq:newd7} from \eqref{eq:newd6},
\begin{equation}
\label{eq:newd8}
\begin{aligned}
 &2m^2[F(x),x]xyx-m(m+n)[F(x),x]x^2y-m(m+n)[F(x),x]yx^2= 0,\\
 &~\text{for all}~x,y\in R.
\end{aligned}
\end{equation}

Put $y=yF(x)$ in \eqref{eq:newd8}, we have
\begin{equation}
\label{eq:newd9}
\begin{aligned}
&2m^2[F(x),x]xyF(x)x-m(m+n)[F(x),x]x^2yF(x)-m(m+n)[F(x),x]yF(x)x^2= 0,\\
 &~\text{for all}~x,y\in R.
\end{aligned}
\end{equation}

Multiplying \eqref{eq:newd8} by $F(x)$ from right, we get
\begin{equation}
\label{eq:newd10}
\begin{aligned}
 &2m^2[F(x),x]xyxF(x)-m(m+n)[F(x),x]x^2yF(x)-m(m+n)[F(x),x]yx^2F(x)= 0,\\
 &~\text{for all}~x,y\in R.
\end{aligned}
\end{equation}

Again, subtract \eqref{eq:newd10} from \eqref{eq:newd9}, we have
\begin{equation}
\label{eq:newd11}
\begin{aligned}
&2m^2[F(x),x]xy[F(x),x]-m(m+n)[F(x),x]y[F(x),x^2]= 0,~\text{for all}~x,y\in R.
\end{aligned}
\end{equation}

By using Lemma \ref{new2} on \eqref{eq:newd11}, we have
\begin{equation}
\label{eq:newd12}
\begin{aligned}
&(2m^2[F(x),x]x-(m^2+mn)[F(x),x^2])y[F(x),x]=0\\
&\implies (-2m^2x[F(x),x]+(m^2-mn)[F(x),x^2])y[F(x),x]=0 ,~\text{for all}~x,y\in R.
\end{aligned}
\end{equation}

Since $[F(x),x^2]=[F(x),x]x+x[F(x),x]$, \eqref{eq:newd12} reduces to
\begin{equation}
\label{eq:newd13}
\begin{aligned}
&((m+n)x[F(x),x]+(n-m)[F(x),x]x)y[F(x),x]=0,\\
&~\text{for all}~x,y\in R~(\text{since $R$ is $m$-torsion free}).
\end{aligned}
\end{equation}

Putting $y=y(m+n)x$ in \eqref{eq:newd13} and multiplying \eqref{eq:newd13} by $(n-m)x$ from right, and then adding, using the semiprimeness of $R$ we have
\begin{equation}
\label{eq:newd14}
\begin{aligned}
&(m+n)x[F(x),x]+(n-m)[F(x),x]x=0\\
&\implies (n-m)F(x)x^2+2mxF(x)x-(m+n)x^2F(x)=0,~\text{for all}~x\in R.
\end{aligned}
\end{equation}

Putting $x=x+y$ in \eqref{eq:newd14},
\begin{equation}
\label{eq:newd15}
\begin{aligned}
&(n-m)F(x)(xy+yx+y^2)+(n-m)F(y)(x^2+xy+yx)\\
&+2mxF(x)y+2mxF(y)x+2mxF(y)y+2myF(x)x+2myF(x)y+2myF(y)x\\
&-(m+n)[(xy+yx+y^2)F(x)+(x^2+xy+yx)F(y)]=0,~\text{for all}~x,y\in R.
\end{aligned}
\end{equation}

Put $x=-x$ in \eqref{eq:newd15} and then adding it to \eqref{eq:newd15}, we have
\begin{equation}
\label{eq:newd16}
\begin{aligned}
&(n-m)F(x)(xy+yx)+(n-m)F(y)x^2+2mxF(x)y+2mxF(y)x+2myF(x)x\\
&-(m+n)[(xy+yx)F(x)+x^2F(y)]=0~\text{for all}~x,y\in R.
\end{aligned}
\end{equation}

Put $y=(m+n)(xy+yx)$ in \eqref{eq:newd16} and using \eqref{eq:newd2},
\begin{equation*}
\begin{aligned}
&(n-m)(m+n)F(x)[x^2y+2xyx+yx^2]\\
&+(n-m)[2mF(x)y+2mF(y)x+2nD(xy)]x^2+2m(m+n)xF(x)(xy+yx)\\
&+2mx[2mF(x)y+2mF(y)x+2nD(xy)]x+2m(m+n)(xy+yx)F(x)x\\
&-(m+n)^2[x^2y+2xyx+yx^2]F(x)-(m+n)x^2[2mF(x)y+2mF(y)x+2nD(xy)]\\
&=0,~\text{for all}~x,y\in R.
\end{aligned}
\end{equation*}

Now, using \eqref{eq:newd14}, we have
\begin{equation*}
\begin{aligned}
&(n-m)(m+n)[x^2F(x)y+2F(x)xyx]
+(n-m)(3m+n)F(x)yx^2\\
&+2(n-m)[mF(y)x^3+nD(xy)x^2]
+2m(3m+n)xF(x)yx+4m^2xF(y)x^2\\
&+4mnxD(xy)x
+(m+n)[2mxyF(x)x-(n-m)yF(x)x^2
-(m+n)x^2yF(x)]\\
&-2(m+n)[(m+n)xyxF(x)-mx^2F(y)x]-2n(m+n)x^2D(xy)\\
&=0,~\text{for all}~x,y\in R.
\end{aligned}
\end{equation*}

Also, applying \eqref{eq:newd16},
\begin{equation}
\begin{aligned}
\label{eq:newd17}
&(n-m)[(m+n)x^2F(x)y+2nF(x)xyx+(m+n)F(x)yx^2]\\
&+2m(m+n)xF(x)yx+4m(m+n)xyF(x)x-(3m^2+n^2)yF(x)x^2\\
&-(m+n)^2x^2yF(x)+2m(m+n)yxF(x)x-2(m+n)^2xyxF(x)\\
&+2n(n-m)D(xy)x^2+4mnxD(xy)x-2n(m+n)x^2D(xy)=0,~\text{for all}~x,y\in R.
\end{aligned}
\end{equation}

Putting $y=xy$ in \eqref{eq:newd17} and also multiplying \eqref{eq:newd17} by $x$ from left and after subtracting these new equations, we get
\begin{equation}
\label{eq:newd18}
\begin{aligned}
&(n-m)(m+n)x^2[F(x),x]y+2n(n-m)[F(x),x]xyx\\
&+(n-m)(m+n)[F(x),x]yx^2+2m(m+n)x[F(x),x]yx=0,~\text{for all}~x,y\in R.
\end{aligned}
\end{equation}

Again, replacing $y$ by $yF(x)$ in \eqref{eq:newd18} and multiplying \eqref{eq:newd18} by $F(x)$ from right, and finally subtracting these obtained relations, we get
\begin{equation*}
\begin{aligned}
&2n(n-m)[F(x),x]xy[F(x),x]+(n-m)(m+n)[F(x),x]y[F(x),x^2]\\
&+2m(m+n)x[F(x),x]y[F(x),x]=0,~\text{for all}~x,y\in R.
\end{aligned}
\end{equation*}

Using \eqref{eq:newd14},
\begin{equation*}
\begin{aligned}
&2(n-m)^2[F(x),x]xy[F(x),x]+(n-m)(m+n)[F(x),x]y[F(x),x^2]=0\\
&~\text{for all}~x,y\in R.
\end{aligned}
\end{equation*}

Now, since
\begin{equation*}
\begin{aligned}
(m+n)[F(x),x^2]&=2m[F(x),x]x+((m+n)x[F(x),x]+(n-m)[F(x),x]x)\\
&=2m[F(x),x]x,~\text{for all}~x\in R~(\text{by} ~\eqref{eq:newd14}),
\end{aligned}
\end{equation*}

we have
\begin{equation*}
\begin{aligned}
&2(n-m)^2[F(x),x]xy[F(x),x]+2m(n-m)[F(x),x]y[F(x),x]x=0\\
&~\text{for all}~x,y\in R.
\end{aligned}
\end{equation*}

By Lemma \ref{new2},
\begin{equation}
\label{eq:newd19}
\begin{aligned}
&[F(x),x]xy[F(x),x]=0,~\text{for all}~x,y\in R~(\text{since $R$ is $2n|n-m|$-torsion free}).
\end{aligned}
\end{equation}

Using \eqref{eq:newd14} in \eqref{eq:newd19},
\begin{equation}
\label{eq:newd20}
\begin{aligned}
&x[F(x),x]y[F(x),x]=0,~\text{for all}~x,y\in R~(\text{since $R$ is $(m+n)$-torsion free}).
\end{aligned}
\end{equation}

Subtracting \eqref{eq:newd20} from \eqref{eq:newd19}, we get
\begin{equation}
\label{eq:newd21}
\begin{aligned}
&[[F(x),x],x]y[F(x),x]=0,~\text{for all}~x,y\in R.
\end{aligned}
\end{equation}

Finally, multiplying \eqref{eq:newd21} by $x$ from right and putting $y=yx$ in \eqref{eq:newd21}, and then subtracting these two, we get
\begin{equation}
\label{eq:newd22}
\begin{aligned}
&[[F(x),x],x]y[[F(x),x],x]=0,~\text{for all}~x,y\in R\\
& \implies [[F(x),x],x]=0,~\text{for all}~x\in R~\text{(since $R$ is semiprime)}.
\end{aligned}
\end{equation}

By Theorem \ref{new3},
\begin{equation}
\label{eq:newd23}
\begin{aligned}
&[F(x),x]=0\\
& \implies F(x)x=xF(x),~\text{for all}~x\in R~\text{(since $R$ is semiprime)}.
\end{aligned}
\end{equation}

Let $\mathcal{F}=F-D$. Then by \eqref{eq:newd23}, $\mathcal{F}$ is an additive mapping satisfying
\begin{equation}
\label{eq:newd24}
\begin{aligned}
\mathcal{F}(x)x=x\mathcal{F}(x)=0,~\text{for all}~x\in R.
\end{aligned}
\end{equation}

Taking $x=x+y$ in \eqref{eq:newd24},
\begin{equation}
\label{eq:newd25}
\begin{aligned}
\mathcal{F}(x)y+\mathcal{F}(y)x=0,~\text{for all}~x,y\in R.
\end{aligned}
\end{equation}

Multiplying \eqref{eq:newd25} by $\mathcal{F}(x)$ from right,
\begin{equation}
\label{eq:newd26}
\begin{aligned}
&\mathcal{F}(x)y\mathcal{F}(x)=0 ~\text{for all}~x,y\in R \\
&\implies \mathcal{F}(x)=0,~\text{for all}~x\in R~\text{(since $R$ is semiprime)}.
\end{aligned}
\end{equation}

Hence $F=D$. Thus, $F$ is a derivation which maps $R$ into $Z(R)$.

\end{proof}

\section*{Acknowledgement}
The authors are thankful to DST, Govt. of India for financial support and Indian Institute of Technology Patna for providing the research facilities.

\end{document}